\documentclass[a4paper,11pt]{article}

\usepackage[top=3.0cm, bottom=3.0cm, inner=3.0cm, outer=3.0cm,
includefoot]{geometry}

\usepackage{verbatim}
\usepackage{caption}
\usepackage{url}
\usepackage{amsmath}
\usepackage{geometry}
\usepackage{amssymb}
\usepackage{amsmath}
\usepackage{graphicx}
\usepackage{amsthm}
\usepackage{bbm}
\usepackage{float}
\usepackage{color,soul}
\usepackage{hyperref}
\usepackage[T1]{fontenc}
\usepackage[utf8]{inputenc}
\usepackage{authblk}
\usepackage{booktabs}
\usepackage{longtable}
\usepackage{enumerate}
\usepackage{tikz}
\usetikzlibrary{shapes.geometric}
\usetikzlibrary{arrows}
\usetikzlibrary{decorations.markings}
\usepackage{standalone}
\usepackage{indentfirst}
\usepackage{mathtools}

\newcommand{\eChar}{\begin{enumerate}[(i)]}
\newcommand{\eCharR}{\begin{enumerate}[(a)]}
\newcommand{\eBr}{\begin{enumerate}[(1)]}

\newcommand\floor[1]{\left\lfloor #1\right\rfloor}
\newcommand\ceil[1]{\left\lceil #1\right\rceil}

\title
{An extremal theorem for graphs with non-negative Bakry--\'Emery curvature}

\author{Jiawei Xie\thanks{Email: y30251370@mail.ecust.edu.cn}}
\author{Zhe You\thanks{Email: y30231280@mail.ecust.edu.cn}}
\affil{School of Mathematics, East China University of Science and Technology}
\date{}

\theoremstyle{plain}
\newtheorem{lemma}{Lemma}[section]
  {%
   \endlemma
   \endgroup}
   
\newtheorem{theorem}[lemma]{Theorem}

\theoremstyle{definition}

\newtheorem{definition}[lemma]{Definition}

\newtheorem{remark}[lemma]{Remark}

\numberwithin{equation}{section}

\begin{document}
\maketitle

\begin{abstract}
Recently, Chen, Liu, and You (An extremal theorem for positive curvature of graphs, arXiv:2607.02297) proved an extremal theorem for  positive Lin--Lu--Yau curvature.
They further proposed the similar problems for other discrete curvature.
In this paper, we prove a sharp extremal theorem  for non-negative Bakry--\'Emery curvature with non-normalized Laplacian: every graph of order \(n\geq 7\) with more than
\[
  \binom {n}{2}-\floor{\frac{n}{2}}-2
\]
edges satisfies $CD(0,\infty)$, and this threshold is optimal.

\end{abstract}

\textbf{Keywords:}  Bakry--\'Emery curvature, extremal graph theory.

\textbf{Mathematics Subject Classification:} 05C35, 53A70

\section{Introduction}
In recent years, many scholars have focused on the intersection of discrete curvature and combinatorial problems.
In particular, the combinatorial structure of graphs with lower Ricci curvature bounds has attracted much attention, such as diameter of amply regular graphs and distance regular graphs~\cite{Chen-CVPDE, Chen-JCTB, Xia}, degree~\cite{Hehl-regular, Hehl25, Bipartite graphs}, connectivity and edge-connectivity~\cite{Horn, CLY, CKL, LX, Bipartite graphs}.

Extremal graph theory is a central area of graph theory. 
One of its core questions is to find the maximum or minimum number of edges a graph can have under given constraints. 
This point of view has yielded many fundamental insights into discrete structures~\cite{Bollobas, Simonovits}. 
Classical illustrations include Turán-type problems, where the constraints forbid certain subgraphs. 
More recently, spectral extremal problems—in which the restrictions are placed on the eigenvalues of a graph—have also received significant attention~\cite{spectral-extremal, Nikiforov}.


An important view is that positive or large discrete Ricci curvature tends to force a graph to have many edges, which motivates us to consider the curvature extremal problems on graphs.
Recently, Chen, Liu, and You~\cite{chen2026ExtremalTheoremPositive} proved an extremal theorem for positive Lin--Lu--Yau curvature, which is a modified version of Ollivier's Ricci curvature~\cite{O09} on graphs introduced by~\cite{LLY11}. 
They further proposed the similar problem on Bakry-\'Emery curvature, which is another discrete Ricci curvature in Bakry-\'Emery's sense~\cite{Bakry} introduced by~\cite{LY10}.
Lin, You, and Zhao~\cite{Bipartite graphs} considered the extremal problem for positive Lin--Lu--Yau curvature on bipartite graphs.
A closely related criterion for positive Lin--Lu--Yau curvature, formulated in terms of forbidden subgraphs in the complement, was obtained in~\cite{CLY26}.

In this paper, we establish an extremal result  for Bakry-\'Emery curvature on graphs with non-normalized Laplacian.

\begin{theorem}\label{main}
 Let $G$ be a graph of order $n\geq 7$.
 If $G$ has more than
\begin{equation}\label{eq:threshold}
     \binom{n}{2}-\floor{\frac{n}{2}}-2
\end{equation}
edges, then $G$ satisfies $CD(0,\infty)$.
\end{theorem}

The threshold in~\eqref{eq:threshold} is sharp for every \(n\geq 7\). 
Moreover, the assumption \(n\geq 7\) is essential. These sharpness issues are discussed in Section~\ref{section:sharpness}.

Theorem~\ref{main} can be equivalently formulated as a criterion for non-negative Bakry-\'Emery curvature in terms of the number of edges in the complement. More
precisely, let \(G\) be a graph of order \(n\geq 7\). 
If the complement
\(\overline{G}\) has at most
\[
  \left\lfloor\frac{n}{2}\right\rfloor+1
\]
edges, then \(G\) has non-negative Bakry-\'Emery curvature.


Throughout the paper, we use the following notation. 
Let $G=(V,E)$ be a simple finite graph with vertex set $V$ and edge set $E$.
Denote by $\overline{G}$ the complement graph of $G$.
For any $x\in V$, let $N_G(x)$ be the set of neighbors of $x$ and let $d_x\coloneqq|N_G(x)|$ be its degree. 
For any two vertices $x$ and $y$, we denote the distance between them by $d(x,y)$.

\section{Preliminaries}
Before introducing Bakry-\'Emery curvature on graphs, we first recall the concept of $\Gamma$-calculus.
\begin{definition}[$\Gamma$ and $\Gamma_2$ operators]
    Let $G=(V, E)$ be a locally finite simple graph. For any two functions $f, g: V \rightarrow \mathbb{R}$, we define
$$
\begin{aligned}
2 \Gamma(f, g) & :=\Delta(f g)-f \Delta g-g \Delta f, \\
2 \Gamma_2(f, g) & :=\Delta \Gamma(f, g)-\Gamma(f, \Delta g)-\Gamma(\Delta f, g) .
\end{aligned}
$$
\end{definition}

For convenience, we write $\Gamma(f):=\Gamma(f, f)$ and $\Gamma_2(f, f)\coloneqq\Gamma_2(f)$.
Note that 
\[
  \Gamma(f,g)(x)
  =\frac12\bigl(\Delta(fg)(x)-f(x)\Delta g(x)-g(x)\Delta f(x)\bigr)
  =\frac12\sum_{y\sim x}(f(y)-f(x))(g(y)-g(x)).
\]
Now we can give the definition of Bakry-\'Emery curvature on graphs as follows.
\begin{definition}[Bakry-\'Emery curvature]\label{def2.2}
     Let $G=(V, E)$ be a locally finite simple graph. Let $K\in \mathbb{R}$ and $N \in(0, \infty]$. We say that a vertex $x \in V$ satisfies the curvature dimension condition $CD(K,N,x)$, if for any $f\in C(V)$, we have
$$\Gamma_2(f)(x) \geq \frac{1}{N}(\Delta f(x))^2+K \Gamma(f)(x) .$$
We call $K$ a lower Ricci curvature bound of $x$, and $N$ a dimension parameter. The graph $G=(V, E)$ satisfies $C D(K, N)$, if all the vertices satisfy $CD(K, N)$. At a vertex $x \in V$, let $K(G, x ; N)$ be the largest $K$ such that curvature dimension condition holds for all functions $f$ at $x$ for a given $N$. We call $K(G, x ; N)$ the Bakry-\'Emery curvature function of $x$.
\end{definition}
Here, we use the notation $\frac{1}{N}=0$ when $N=\infty$.
For any locally finite graph $G$, the  non-normalized graph Laplacian $\Delta$ is defined as $$\Delta f(x)\coloneqq \sum_{y: xy\in E(G)}(f(y)-f(x)), \text{ for any $f: V(G)\to \mathbb{R}$ and any $x\in V(G)$}.$$

Let $\lambda_{max}(G)$ be the largest Laplacian eigenvalue of graph $G$.
We recall an estimate for $\lambda_{max}(G)$ as follows.
\begin{lemma}[{\cite[Theorem 2]{LMA}}]\label{lambda<m+1}
    For any graph $G$, $\lambda_{max}(G)\leq \max\{d_u+d_v:uv\in E(G)\}$.
\end{lemma}

\section{\texorpdfstring{Proof of main theorem}{Proof of the main theorem}}

In this section, we prove Theorem~\ref{main}.
By the assumption of the theorem, there are at most $ \left\lfloor \frac n2\right\rfloor+1$ edges in the complement graph $\overline{G}$ of $G$.

Fix a vertex $x\in V(G)$.
Let $S=N_G(x)$ and $R=V(G)\setminus(S\cup\{x\})$.
Set $|S|=d$ and $|R|=r$.
Since the terms  $\Gamma(f),\Gamma_2(f)$ and $\Delta f$ are all invariant by adding a constant to $f$,
it suffices to consider $f(x)=0$.

\begin{lemma}\label{P+Q}
    For any $f\in C(V)$ such that $f(x)=0$, we have
    \[
  2\Gamma_2(f)(x)=P_x(f)+\sum_{z\in R}Q_{x,z}(f),
\]
where
\begin{align}
  P_x(f)
  &=\left(\sum_{u\in S}f(u)\right)^2
    +\frac{3-d}{2}\sum_{u\in S}f(u)^2
    +2\sum_{uv\in E_G(S)}(f(u)-f(v))^2,\label{eq:P-def}\\
  Q_{x,z}(f)
  &=\sum_{u\in S\cap N_G(z)}
    \left(\frac12 f(z)^2-2f(u)f(z)+\frac32 f(u)^2\right).\label{eq:Q-def}
\end{align}
Here $E_G(S)$ denotes the set of edges of induced subgraph $G[S]$ in graph $G$.
\end{lemma}

 \begin{proof}
 Since $f(x)=0$, we have
\[
  \Gamma(f)(x)=\frac12\sum_{u\in S}f(u)^2
  \quad \text{ and }\quad
  \Delta f(x)=\sum_{u\in S}f(u).
\]
Thus,
\begin{align*}
  \Delta\Gamma(f)(x)
  &=\sum_{u\in S}\bigl(\Gamma(f)(u)-\Gamma(f)(x)\bigr)  \\
  &=\frac12\sum_{u\in S}\sum_{w\sim u}(f(w)-f(u))^2
    -\frac d2\sum_{u\in S}f(u)^2.
\end{align*}
Also we obtain
\[
  2\Gamma(f,\Delta f)(x)
  =\sum_{u\in S}f(u)\bigl(\Delta f(u)-\Delta f(x)\bigr)
  =\sum_{u\in S}f(u)\Delta f(u)-\left(\sum_{u\in S}f(u)\right)^2.
\]
By the definition,  we have
\begin{align*}
 2\Gamma_2(f)(x)&=\Delta\Gamma(f)(x)-2\Gamma(f,\Delta f)(x)\\
  &=\frac{1}{2}\sum_{u\in S}\sum_{w\sim u}(f(w)-f(u))^2
   -\frac{d}{2}\sum_{u\in S}f(u)^2
   -\sum_{u\in S}f(u)\Delta f(u)
   +\left(\sum_{u\in S}f(u)\right)^2\\
   &=\left(\sum_{u\in S}f(u)\right)^2
    -\frac {d}{2}\sum_{u\in S}f(u)^2+I_x+I_S+I_R,
\end{align*}
where 
\begin{align*}
     I_x &:=\sum_{u\in S}
    \left[\frac12(f(x)-f(u))^2-f(u)(f(x)-f(u))\right],\\
  I_S &:=\sum_{u\in S}\sum_{\substack{v\in S\\ uv\in E(G)}}
    \left[\frac12(f(v)-f(u))^2-f(u)(f(v)-f(u))\right],\\
  I_R &:=\sum_{u\in S}\sum_{\substack{z\in R\\ uz\in E(G)}}
    \left[\frac12(f(z)-f(u))^2-f(u)(f(z)-f(u))\right].
\end{align*}
Observe that $I_x$, $I_S$, and $I_R$ are obtained by dividing $\frac{1}{2}\sum_{u\in S}\sum_{w\sim u}(f(w)-f(u))^2-\sum_{u\in S}f(u)\Delta f(u)$ into three parts according to the neighbors of $u\in S$.
By direct calculation, we have
\begin{align*}
  I_x
  =\sum_{u\in S}\left(\frac{1}{2} f(u)^2+f(u)^2\right)
  =\frac{3}{2}\sum_{u\in S}f(u)^2,
\end{align*}
\begin{align*}
  I_S
  &=\sum_{uv\in E_G(S)}
    \left[\frac{1}{2}(f(v)-f(u))^2-f(u)(f(v)-f(u))\right]  \\
  &\quad+\sum_{vu\in E_G(S)}
    \left[\frac12(f(u)-f(v))^2-f(v)(f(u)-f(v))\right]  \\
  &=\sum_{uv\in E_G(S)}
    \left[(f(u)-f(v))^2+f(u)^2+f(v)^2-2f(u)f(v)\right]  \\
  &=2\sum_{uv\in E_G(S)}(f(u)-f(v))^2,
\end{align*}
and
\begin{align*}
  I_R
  &=\sum_{z\in R}\sum_{u\in S\cap N_G(z)}
    \left[\frac{1}{2}(f(z)-f(u))^2-f(u)(f(z)-f(u))\right]  \\
  &=\sum_{z\in R}\sum_{u\in S\cap N_G(z)}
    \left[\frac{1}{2} f(z)^2-2f(u)f(z)+\frac{3}{2} f(u)^2\right].
\end{align*}
Combined with $I_x$, $I_R$, and $I_R$, we obtain
\begin{align*}
  2\Gamma_2(f)(x)
  &=\left(\sum_{u\in S}f(u)\right)^2
    +\frac{3-d}{2}\sum_{u\in S}f(u)^2
    +2\sum_{uv\in E_G(S)}(f(u)-f(v))^2 \\
  &\quad
    +\sum_{z\in R}\sum_{u\in S\cap N_G(z)}
    \left(\frac12 f(z)^2-2f(u)f(z)+\frac32 f(u)^2\right)\\
    &=P_x(f)+\sum_{z\in R}Q_{x,z}(f).
\end{align*}
The proof is complete.
 \end{proof}

\begin{remark}
The vertices which really influence $\Gamma_2(f)(x)$ are $x$, neighbors of $x$, and 2-neighbors $S_2(x)\coloneqq \{z|\ d(x,z)=2\}$ of $x$.
Thus, the second term $\sum_{z\in R}Q_{x,z}(f)$ is equal to  $\sum_{z\in S_2(x)}Q_{x,z}(f)$.
\end{remark}
Our aim is to estimate $P_x(f)$ and $\sum_{z\in R}Q_{x,z}(f)$ respectively.
Denote $S_z=N_G(z)\cap S$.
\begin{lemma}\label{Q(f)}
   For any $z\in R$ and $f\in C(V)$ such that $f(x)=0$, we have
\[
  Q_{x,z}(f)\ge -\frac12\sum_{u\in S_z}f(u)^2.
\]
Moreover,
\[
  \sum_{z\in R}Q_{x,z}(f)
  \ge -\frac r2\sum_{u\in S}f(u)^2.
\]
\end{lemma}

\begin{proof}
  If $S_z=\emptyset$, then $Q_{x,z}(f)=0$.
  Assume that $|S_z|>0$.
  Then
  \begin{align*}
      Q_{x,z}(f) &=\sum_{u\in S_z}
    \left(\frac{1}{2} f(z)^2-2f(u)f(z)+\frac32 f(u)^2\right)\\
      &=\frac{|S_z|}{2}f(z)^2 -2f(z)\sum_{u\in S_z}f(u) +\frac{3}{2}\sum_{u\in S_z}f(u)^2\\
      &=\frac{|S_z|}{2}
    \left(f(z)-\frac{2\sum_{u\in S_z}f(u)}{|S_z|}\right)^2
    +\frac{3}{2}\sum_{u\in S_z}f(u)^2
    -\frac{2\left(\sum_{u\in S_z}f(u)\right)^2}{|S_z|}\\
    &\geq \frac{3}{2}\sum_{u\in S_z}f(u)^2
    -\frac{2\left(\sum_{u\in S_z}f(u)\right)^2}{|S_z|}.
  \end{align*}
By Cauchy--Schwarz inequality, we have
\[
  \left(\sum_{u\in S_z}f(u)\right)^2
  \le |S_z|\sum_{u\in S_z}f(u)^2.
\]
Therefore,
\[
  Q_{x,z}(f)
  \ge \frac32\sum_{u\in S_z}f(u)^2
      -2\sum_{u\in S_z}f(u)^2
  =-\frac12\sum_{u\in S_z}f(u)^2.
\]
Summing over $z\in R$, we get
\[
  \sum_{z\in R}Q_{x,z}(f)
  \ge -\frac12\sum_{z\in R}\sum_{u\in S_z}f(u)^2
  = -\frac12\sum_{u\in S}|N_G(u)\cap R|\,f(u)^2.
\]
Since $|N_G(u)\cap R|\le r$, we have
\[
  \sum_{z\in R}Q_{x,z}(f)
  \ge -\frac r2\sum_{u\in S}f(u)^2.
\]
The proof is complete.
\end{proof}

\begin{lemma}\label{P(f)}
    $P_x(f)\ge \alpha\sum_{u\in S}f(u)^2,$
where
$\alpha=\min\left\{\frac{d+3}{2},\frac{3+3d}{2}-2(h+1)\right\}$.
\end{lemma}
\begin{proof}
    If $S=\emptyset$, there is nothing to prove.
   Assume that $d\geq 1$.
   Let $H\coloneqq \overline{G}[S]$.
Hence $E_G(S)=\binom{S}{2}\setminus E(H)$.
Therefore, we have
\begin{align}
  \sum_{uv\in E_G(S)}(f(u)-f(v))^2
  &=\sum_{\{u,v\}\in \binom{S}{2}}(f(u)-f(v))^2
    -\sum_{uv\in E(H)}(f(u)-f(v))^2 \nonumber\\
&=\frac{1}{2}\sum_{u,v\in S}(f(u)-f(v))^2 -\sum_{uv\in E(H)}(f(u)-f(v))^2 \nonumber\\
&=\frac{1}{2}\left(2d\sum_{u\in S}f(u)^2-2\sum_{u,v\in S}f(u)f(v) \right) -\sum_{uv\in E(H)}(f(u)-f(v))^2 \nonumber\\
  &=d\sum_{u\in S}f(u)^2
    -\left(\sum_{u\in S}f(u)\right)^2
    -\sum_{uv\in E(H)}(f(u)-f(v))^2 .\label{eq:EGS-H-expanded}
\end{align}
Substituting~\eqref{eq:EGS-H-expanded} into~\eqref{eq:P-def}, we have
\begin{align}
  P_x(f)
  &=\left(\sum_{u\in S}f(u)\right)^2
    +\frac{3-d}{2}\sum_{u\in S}f(u)^2
    +2\sum_{uv\in E_G(S)}(f(u)-f(v))^2 \nonumber\\
  &=\left(\sum_{u\in S}f(u)\right)^2
    +\frac{3-d}{2}\sum_{u\in S}f(u)^2 \nonumber\\
  &\quad
    +2d\sum_{u\in S}f(u)^2
    -2\left(\sum_{u\in S}f(u)\right)^2
    -2\sum_{uv\in E(H)}(f(u)-f(v))^2 \nonumber\\
  &=\frac{3+3d}{2}\sum_{u\in S}f(u)^2
    -\left(\sum_{u\in S}f(u)\right)^2
    -2\sum_{uv\in E(H)}(f(u)-f(v))^2 .\label{eq:P-matrix-before}
\end{align}
Define the inner product on vertex subset $S$ as $\langle \phi,\psi\rangle_S\coloneqq\sum_{u\in S}\phi(u)\psi(u)$.
It follows that
\[
  \langle \phi,J\psi\rangle_S
  =\left(\sum_{u\in S}\phi(u)\right)
   \left(\sum_{u\in S}\psi(u)\right)
\]
and
\[
  \langle \phi,L_H\psi\rangle_S
  =\sum_{uv\in E(H)}(\phi(u)-\phi(v))(\psi(u)-\psi(v)),
\]
where $J$ denotes the all-one matrix and $L_H$ denotes the Laplacian matrix on $H$.
By~\eqref{eq:P-matrix-before}, we have
\begin{equation}\label{eq:P-operator-form-detailed}
  P_x(f)
  =\left\langle f|_S,
  \left(\frac{3+3d}{2}I-J-2L_H\right)f|_S
  \right\rangle_S .
\end{equation}
Decompose \(f|_S\) orthogonally as
$f|_S=c\mathbf 1+g$, where $ g\perp {\mathbf 1}$.
Thus $\langle c{\bf 1},g\rangle_S=0$
and
  $\|f|_S\|_S^2=\|c{\bf 1}\|_S^2+\|g\|_S^2
  =\sum_{u\in S}f(u)^2$.
Let $A\coloneqq\frac{3+3d}{2}I-J-2L_H$.
Combining with
$J{\bf 1}=d{\bf 1}$, $Jg=\left(\sum_{u\in S}g(u)\right){\bf 1}=0$, and $L_H{\bf 1}=0$,
we have
\begin{align}
  \langle c{\bf 1},A(c{\bf 1})\rangle_S
  &=\frac{3+3d}{2}\|c{\bf 1}\|_S^2
    -\langle c{\bf 1},J(c{\bf 1})\rangle_S
    -2\langle c{\bf 1},L_H(c{\bf 1})\rangle_S \nonumber\\
  &=\frac{3+3d}{2}\|c{\bf 1}\|_S^2
    -d\|c{\bf 1}\|_S^2 \nonumber\\
  &=\frac{d+3}{2}\|c{\bf 1}\|_S^2,\label{eq:constant-direction-detailed}
  \end{align}
  and
  \begin{align}
  \langle c{\bf 1},Ag\rangle_S
  &=\frac{3+3d}{2}\langle c{\bf 1},g\rangle_S
    -\langle c{\bf 1},Jg\rangle_S
    -2\langle c{\bf 1},L_Hg\rangle_S \nonumber\\
  &=-2\langle L_H(c{\bf 1}),g\rangle_S
   =0.\label{eq:cross-term-detailed}
\end{align}
Denote $|E(H)|$ by $h$.
By Rayleigh quotient and Lemma~\ref{lambda<m+1}, 
\begin{align}
  \langle g,Ag\rangle_S
  &=\frac{3+3d}{2}\|g\|_S^2
    -\langle g,Jg\rangle_S
    -2\langle g,L_Hg\rangle_S \nonumber\\
  &=\frac{3+3d}{2}\|g\|_S^2
    -2\langle g,L_Hg\rangle_S \nonumber\\
  &\ge \frac{3+3d}{2}\|g\|_S^2
    -2\lambda_{\max}(H)\|g\|_S^2 \nonumber\\
  &\ge \left(\frac{3+3d}{2}-2(h+1)\right)\|g\|_S^2 .\label{eq:orthogonal-direction-detailed}
\end{align}
The last inequality holds since $\lambda_{max}(H)\leq \max\{d_u+d_v:uv\in E(H)\}
\leq h+1$.
By~\eqref{eq:P-operator-form-detailed},~\eqref{eq:constant-direction-detailed},~\eqref{eq:cross-term-detailed}, and~\eqref{eq:orthogonal-direction-detailed},
\begin{align*}
  P_x(f)
  &=\langle c{\bf 1}+g,A(c{\bf 1}+g)\rangle_S \\
  &=\langle c{\bf 1},A(c{\bf 1})\rangle_S
    +2\langle c{\bf 1},Ag\rangle_S
    +\langle g,Ag\rangle_S \\
  &\ge \frac{d+3}{2}\|c{\bf 1}\|_S^2
    +\left(\frac{3+3d}{2}-2(h+1)\right)\|g\|_S^2 \\
  &\ge
  \min\left\{\frac{d+3}{2},\frac{3+3d}{2}-2(h+1)\right\}
  \bigl(\|c{\bf 1}\|_S^2+\|g\|_S^2\bigr) \\
  &=\alpha\sum_{u\in S}f(u)^2 .
\end{align*}
The proof is complete.
\end{proof}

Now we are ready to prove our main theorem.
\begin{proof}[Proof of Theorem~\ref{main}]
Let $f(x)=0$.
    By the assumption of the theorem, there are at most $ \left\lfloor \frac n2\right\rfloor+1$ edges in the complement graph $\overline{G}$ of $G$.
    Hence $r\leq r+h\leq \left\lfloor \frac n2\right\rfloor+1$.
    By Lemma~\ref{P+Q}, Lemma~\ref{Q(f)}, and Lemma~\ref{P(f)}, it suffices to prove $\alpha\geq \frac{r}{2}$.

 Note that $d+r=n-1$. 
 So 
 $$\frac{d+3}{2}-\frac{r}{2}=\frac{n+2-2r}{2}\geq0.$$
It remains to prove $\frac{3+3d}{2}-2(h+1)\geq\frac{r}{2}$.
Indeed,
\begin{align*}
  \frac{3+3d}{2}-2(h+1)-\frac r2
  &=\frac{3+3(n-1-r)}{2}-2h-2-\frac r2  \\
  &=\frac{3n-4r}{2}-2h-2\\
  &\ge \frac{3n-4r}{2}-2(\floor{\frac{n}{2}}+1-r)-2  \\
  &=\frac{3n}{2}-2\floor{\frac{n}{2}}-4.
\end{align*}
If $n=2m$, then $m\ge 4$, and $\frac{3n}{2}-2\floor{\frac{n}{2}}-4=m-4\geq 0$.
If $n=2m+1$, then $m\ge 3$, and
$\frac{3n}{2}-2\floor{\frac{n}{2}}-4 =m-\frac{5}{2}\ge 0.$
Therefore, $ \frac{3+3d}{2}-2(h+1)\ge \frac{r}{2}$.

Overall, $\alpha\ge \frac{r}{2}$, which implies $x$ satisfies $CD(0,\infty;x)$.
Since $x$ is chosen arbitrarily, $G$ satisfies $CD(0,\infty)$.
\end{proof}
   
\section{Sharpness discussions}\label{section:sharpness}

In this section, we present an example to illustrate that the bound $\binom{n}{2}-\floor{\frac{n}{2}}-2$ in Theorem~\ref{main} is sharp.
Let $G=(V,E)$ be the graph  of order $n$ defined as follows.

The vertex set \[V\coloneqq\{x\}\sqcup N_G(x)\sqcup R,\] where $|N_G(x)|=\ceil{\frac{n}{2}}-3$ and $|R|=\floor{\frac{n}{2}}+2$.
The edge set 
$$E\coloneqq E_1\sqcup  E_2\sqcup E_3,$$ where 
$E_1\coloneqq\{uv\ |\ u, v\in N_G(x)\cup \{x \} \text{ and } u\neq v\}$, $E_2\coloneqq \{uv\ | \ u\in N_G(x),\ v\in R\}$, and $E_3\coloneqq\{uv\ |\ u,v\in R \text{ and } u\neq v\}$.
Then, $G$ has exactly $\binom {n}{2}-\floor{\frac{n}{2}}-2$ edges.
A schematic of $G$ is depicted in Figure \ref{fig:enter-label}.
\begin{figure}[htbp]
    \centering    \begin{tikzpicture}[scale=1.15, line cap=round, line join=round]

\node[circle, fill, inner sep=2.1pt, label=left:{\Large \(x\)}] (x) at (-3,0) {};

\node[
  draw,
  ellipse,
  thick,
  minimum width=2.0cm,
  minimum height=2.45cm,
  inner sep=0pt
] (A) at (0,0) {\Large \(K_{\left\lceil \frac{n}{2}\right\rceil-3}\)};

\node[
  draw,
  ellipse,
  thick,
  minimum width=2.0cm,
  minimum height=2.45cm,
  inner sep=0pt
] (B) at (3.2,0) {\Large \(K_{\left\lfloor\frac{n}{2}\right\rfloor+2}\)};

\draw (x.30)  -- (A.150);
\draw (x.10)  -- (A.175);
\draw (x.-10) -- (A.185);
\draw (x.-30) -- (A.210);

\draw (A.35)  -- (B.145);
\draw (A.12)  -- (B.168);
\draw (A.-12) -- (B.192);
\draw (A.-35) -- (B.215);

\end{tikzpicture}
    \caption{A schematic of the graph $G$.}
    \label{fig:enter-label}
\end{figure}
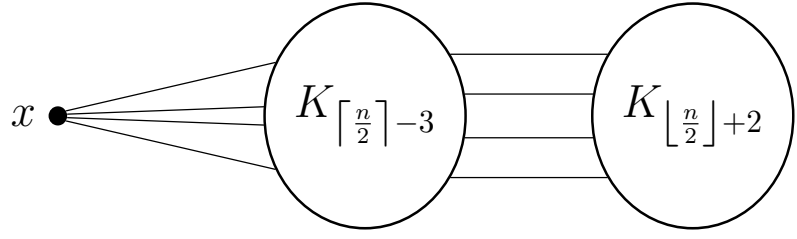

Finally, we present a graph of order $n=6$ with exactly $\binom {n}{2}-\floor{\frac{n}{2}}-1$ edges such that some edges have negative Bakry-\'Emery curvature, as shown in Figure~\ref{n=6}.
The curvature on each edge was computed using the graph curvature calculator~\cite{CKLLS22}, which is a freely accessible interactive app at 
https://www.mas.ncl.ac.uk/graph-curvature/. Therefore, the condition $n\geq 7$ in Theorem~\ref{main} is necessary.

 \begin{figure}[ht]
     \centering
\includegraphics[width=0.6\textwidth,height=0.35\textwidth]{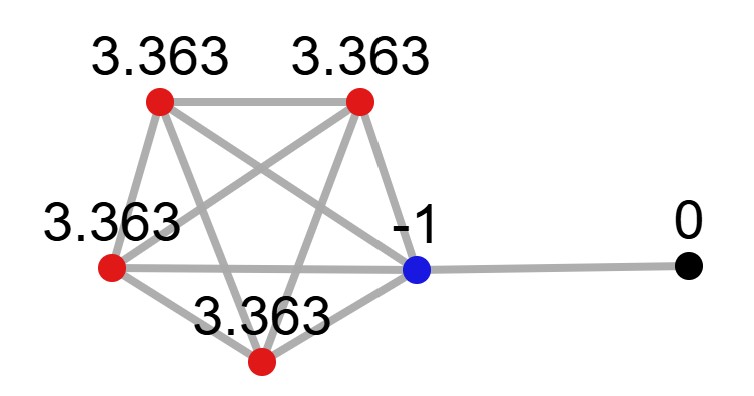}
     \caption{A graph of order 6 containing vertices of negative Bakry--\'Emery curvature.}
     \label{n=6}
 \end{figure}


\end{document}